\documentclass[a4,12pt]{article}

\usepackage{amsfonts, amsmath, amssymb, amsgen, amsthm, amscd,latexsym,mathrsfs}

\usepackage{color}
\usepackage[all]{xy}

\usepackage[top=3.2cm, left=2.2cm, bottom=2.2cm, right=2.2cm]{geometry}

\def\Diff{\mathop{\rm Diff}\nolimits}

\def\Ad{\mathop{\rm Ad}\nolimits}

\def\log{\mathop{\rm log}\nolimits}

\def\cotor{\mathop{\rm Cotor}\nolimits}

\def\Rb{{\mathbb R}}

\def\Zb{{\mathbb Z}}

\def\Ac{{\cal A}}

\def\Fc{{\cal F}}

\def\Hc{{\cal H}}

\def\Uc{{\cal U}}

\def\Cc{{\cal C}}

\def\a{\alpha}
\def\b{\beta}
\def\d{\delta}
\def\D{\Delta}

\def\s{\sigma}

\def\ve{\varepsilon}
\def\vp{\varphi}

\def\0b{\bf 0}

\def\nb{\nabla}
\def\ot{\otimes}

\def\ra{\rightarrow}

\def\0D{\Delta^{(0)}}
\def\1D{\Delta^{(1)}}
\def\Db{\blacktriangledown}

\def\wg{\wedge}

\newcommand{\wbar}[1]{\overline{#1}}

\newcommand{\Fb}{\mathfrak{b}}

\newcommand{\Fg}{\mathfrak{g}}
\newcommand{\Fh}{\mathfrak{h}}

\newtheorem{theorem}{Theorem}[section]
\newtheorem{remark}[theorem]{Remark}
\newtheorem{proposition}[theorem]{Proposition}
\newtheorem{lemma}[theorem]{Lemma}
\newtheorem{corollary}[theorem]{Corollary}

\def\build#1_#2^#3{\mathrel{
\mathop{\kern 0pt#1}\limits_{#2}^{#3}}}
\newcommand{\ps}[1]{~\hspace{-4pt}_{^{(#1)}}}

\newcommand{\ns}[1]{~\hspace{-4pt}_{_{{<#1>}}}}

\def\odots{\ot\cdots\ot}

\def\one{{\bf 1}}

\newcommand{\ie}{{\it i.e.\/}\ }

\def\a{\alpha}
\def\b{\beta}
\def\d{\delta}

\def\s{\sigma}

\def\ve{\varepsilon}

\def\vp{\varphi}

\def\D{\Delta}

\def\nb{\nabla}

\def\ot{\otimes}
\def\part{\partial}

\def\ra{\rightarrow}

\def\lra{\leftrightarrow}
\def\text{\hbox}

\def\nb{\nabla}
\def\ot{\otimes}

\def\ra{\rightarrow}

\def\Ad{\mathop{\rm Ad}\nolimits}

\def\Diff{\mathop{\rm Diff}\nolimits}

\def\lra{\longrightarrow}

\def\build#1_#2^#3{\mathrel{
\mathop{\kern 0pt#1}\limits_{#2}^{#3}}}

\newcommand{\CH}{\text{\bf CH}}

\newcommand{\CB}{\text{\bf CB}}

\numberwithin{equation}{section}
\title{Hopf-cyclic Cohomology of Quantum Enveloping Algebras}
\author{Atabey Kaygun and Serkan S\"utl\"u}

\date{}

\begin{document}
\maketitle

\begin{abstract}
In this paper we calculate both the periodic and non-periodic
Hopf-cyclic cohomology of Drinfeld-Jimbo quantum enveloping algebra
$U_q(\mathfrak{g})$ for an arbitrary semi-simple Lie algebra
$\mathfrak{g}$ with coefficients in a modular pair in involution.
We show that its Hochschild cohomology is concentrated in a single
degree determined by the rank of the Lie algebra $\mathfrak{g}$.
\end{abstract}


\section{Introduction}

In this paper we calculate the Hopf-cyclic cohomology of
Drinfeld-Jimbo quantum enveloping algebra $U_q(\mathfrak{g})$ for an
arbitrary semi-simple Lie algebra $\mathfrak{g}$ with coefficients in
a modular pair in involution (MPI) ${}^\sigma k_\ve$.  This cohomology
was previously calculated only for $s\ell_2$ by Crainic
in~\cite{Crai02}.  We also verified the original calculations of
Moscovici-Rangipour~\cite{MoscRang07} of the Hopf-cyclic cohomology of
the Connes-Moscovici Hopf algebras $\Hc_1$ and $\Hc_{\rm 1S}$ with
coefficients in the trivial MPI ${}^1k_\ve$ using our new
cohomological machinery.

\medskip

The calculation of the Hopf-cyclic cohomology for Connes-Moscovici
Hopf algebras $\Hc_n$ is a big challenge.  These Hopf algebras are
designed to calculate the characteristic classes of codimension-$n$
foliations~\cite{ConnMosc98}.  The intricate calculations of
Hopf-cyclic cohomology of these Hopf algebras by Moscovici and
Rangipour used crucially the fact that these Hopf algebras are
bicrossed product Hopf algebras~\cite{Majid-book, HadfMaji07,
MoscRang07,MoscRang09,MoscRang11}.  On top of previously
calculated explicit classes of $\Hc_1$, explicit representatives of
Hopf-cyclic cohomology classes of $\Hc_2$ are recently obtained in
\cite{RangSutl-IV} via a cup product with a SAYD-twisted cyclic
cocycle.  More recently, Moscovici gave a geometric approach for
$\Hc_n$ in \cite{Mosc14} using explicit quasi-isomorphisms
between the Hopf-cyclic complex of $\Hc_n$, Dupont's simplicial de
Rham DG-algebra~\cite{Dupo76}, and the Bott complex~\cite{Bott76}.

\medskip

The cohomological machinery we developed in this paper, allowed us to
replicate the results of \cite{MoscRang07} on the Hopf-cyclic
cohomology of the Hopf algebra $\Hc_1$, and its Scwarzian quotient
$\Hc_{\rm 1S}$.  To this end, we start by calculating in Proposition
\ref{prop-weight-1-Hochschild} the cohomologies of the weight 1
subcomplex of the coalgebra Hochschild complexes of these Hopf
algebras using their canonical grading, but without appealing their
bicrossed product structure.  Then we use the Cartan homotopy formula
for $\Hc_1$, as developed in \cite{MoscRang07}, to obtain the periodic
Hopf-cyclic cohomology groups with coefficients in the trivial, which
happens to be the only finite dimensional, SAYD module for these Hopf
algebras~\cite{RangSutl-III}.

\medskip

On the quantum enveloping algebra side there are several computations
\cite{ParsWang92, ParsWang93, GinzKuma93,Hu99} on the ${\rm
  Ext}$-groups.  However, the literature on Hopf-cyclic cohomology, or
even coalgebraic cohomology of any variant, of quantum enveloping
algebras is rather meek.  The only results we are aware of are both
for $U_q(s\ell_2)$: one for the ordinary Hopf-cyclic cohomology by
Crainic~\cite{Crai02}, and one for the dual Hopf-cyclic cohomology by
Khalkhali and Rangipour~\cite{KhalRang02}.

\medskip

The central result we achieve in this paper is the computation of both
the periodic and non-periodic Hopf-cyclic cohomology of the quantized
enveloping algebras $U_q(\Fg)$ in full, for an arbitrary semisimple
Lie algebra $\Fg$. In Theorem~\ref{prop-C-Hochschild} we first
calculate the coalgebra Hochschild cohomology of $U_q(\mathfrak{g})$
with coefficients in the comodule ${}^\s k_\ve$ of
Klimyk-Schm\"udgen~\cite[Prop. 6.6]{KlimSchm-book}, which is in fact
an MPI over $U_q(\mathfrak{g})$.  We observe that the Hochschild
cohomology is concentrated in a single degree determined by the rank
of the Lie algebra $\mathfrak{g}$, and finally we calculate the
periodic and non-periodic Hopf-cyclic cohomology groups of
$U_q(\mathfrak{g})$ in Theorem~\ref{thm-Uq-g-hopf-cyclic}.

\medskip

One of the important implications of Theorem~\ref{prop-C-Hochschild}
is that we now have candidates for noncommutative analogues of the
Haar functionals for $U_q(\mathfrak{g})$. The fact that coalgebra
Hochschild cohomology of $U_q(s\ell_2)$ is concentrated only in a
single degree was first observed by Crainic in~\cite{Crai02}.  The
dual version of the statement, that is the algebra Hochschild homology
of $k_q[SL(N)]$ the quantized coordinate ring of $SL(N)$ with
coefficients twisted by the modular automorphism $\sigma$ of the Haar
functional is also concentrated in a single degree, is proven by
Hadfield and Kraehmer in~\cite{HadfKrae06}.  Kraehmer used this fact
to prove an analogue of the Poicare duality for Hochschild homology
and cohomology for $k_q[SL(N)]$ in~\cite{Krae07}.  We plan on
investigating the ramification of the fact that Hochschild cohomology
of $U_q(\Fg)$ is concentrated in a single degree, and its connections
with the dimension-drop phenomenon and twisted Calabi-Yau coalgebras
~\cite{HeTorrOystZhan11}, in a future paper.

\section{Preliminaries}

In this section we recall basic material that will be needed in the
sequel. More explicitly, in the first subsection our objective is to
recall the coalgebra Hochschild cohomology. To this end we also bring
the definitions of the cobar complex of a coalgebra, and hence the
$\cotor$-groups. The second subsection, on the other hand, is devoted to
a very brief summary of the Hopf-cyclic cohomology with coefficients.

\subsection{Cobar and Hochschild complexes}

In this subsection we recall the definition of the cobar complex of a coalgebra $\Cc$, and it is followed by the definition of the $\cotor$-groups associated to a coalgebra $\Cc$ and a pair $(V,W)$ of $\Cc$-comodules of opposite parity.

\medskip

Let $\Cc$ be a coassociative coalgebra. Following \cite{BrzeWisb-book,Doi81} and \cite{KaygKhal06}, the cobar complex of
$\Cc$ is defined to be the differential graded space
\begin{equation*}
{\CB}^\ast(\Cc) := \bigoplus_{n\geq0}\Cc^{\ot n+2}
\end{equation*}
with the differential
\begin{align*}
\begin{split}
& d:{\CB}^n(\Cc) \longrightarrow {\CB}^{n+1}(\Cc)\\
& d(c^0\odots c^{n+1})=\sum_{j=0}^{n}(-1)^j\,c^0\odots \D(c^j)\odots c^{n+1}.
\end{split}
\end{align*}

Let $\Cc^e:= \Cc \ot \Cc^{\rm cop}$ be the enveloping coalgebra of $\Cc$. In case $\Cc$ is counital, the cobar complex ${\CB}^\ast(\Cc)$ yields a $\Cc^e$-injective resolution of the (left) $\Cc^e$-comodule $\Cc$, \cite{Doi81}.

\medskip

Following the terminology of \cite{Kayg12}, for a pair $(V,W)$ of two $\Cc$-comodules of opposite parity (say, $V$ is a right $\Cc$-comodule, and $W$ is a left $\Cc$-comodule), we call the complex
\begin{equation*}
\left({\CB}^\ast(V,\,\Cc,\,W),\,d\right), \quad {\CB}^\ast(V,\,\Cc,\,W):=V\Box_\Cc {\CB}^\ast(\Cc)\Box_\Cc W
\end{equation*}
where
\begin{align}\label{aux-cobar}
\begin{split}
& d:{\CB}^n(V,\,\Cc,\,W)\lra {\CB}^{n+1}(V,\,\Cc,\,W),\\
& d(v \ot c^1 \ot \ldots\ot c^n \ot w) = \\
& v\ns{0}\ot v\ns{1}\ot c^1\odots c^{n}\ot w + \sum_{j=1}^{n}(-1)^j\,c^1\odots \D(c^j)\odots c^n \ot w \\
&+ (-1)^{n+1}\,v\ot c^1\odots c^{n}\ot w\ns{-1}\ot w\ns{0},
\end{split}
\end{align}
the two-sided (cohomological) cobar complex of the coalgebra $\Cc$.

\medskip

The $\cotor$-groups of a pair $(V,W)$ of $\Cc$-comodules of opposite parity are defined by
\begin{equation*}
\cotor_\Cc^\ast(V,W):=H_\ast(V\Box_\Cc \wbar{Y}(\Cc)\Box_{\Cc}W,\,d),
\end{equation*}
where $\wbar{Y}(\Cc)$ is an injective resolution of $\Cc$ via $\Cc$-bicomodules. In case $\Cc$ is a counital coalgebra one has
\begin{equation}\label{aux-cotor-by-cobar}
\cotor_\Cc^\ast(V,W)=H_\ast({\CB}^\ast(V,\,\Cc,\,W),\,d).
\end{equation}

\medskip

We next recall the Hochschild cohomology of a coalgebra $\Cc$ with coefficients in the $\Cc$-bicomodule (equivalently $\Cc^e$-comodule) $V$, from \cite{Doi81}, as the homology of the complex
\begin{equation*}
\CH^\ast(\Cc,V) = \bigoplus_{n\geq 0} \CH^n(\Cc,V),\qquad\CH^n(\Cc,V):=V\ot \Cc^{\ot \,n}
\end{equation*}
with the differential
\begin{align}\label{aux-Hochschild-differential}
\begin{split}
& b:\CH^n(\Cc,V) \to \CH^{n+1}(\Cc,V)\\
& b(v\ot c^1\odots c^n) =\\
 &v\ns{0}\ot v\ns{1} \ot c^1 \odots c^n + \sum_{k= 1}^n(-1)^kc^1\odots \D(c^k)\odots c^n\\
 & + (-1)^{n+1} v\ns{0}\ot c^1\odots c^n\ot v\ns{-1}.
\end{split}
\end{align}
Identification
\begin{equation*}
\CB^n(\Cc) \cong \Cc^e \ot \Cc^{\ot \,n},\,n>0,\quad c^0\odots c^{n+1} \lra (c^0\ot c^{n+1})\ot c^1\odots c^n
\end{equation*}
as left $\Cc^e$-comodules, where the left $\Cc^e$-comodule structure on $\Cc^{\ot\,n+2}$ is given by $\nb(c^0\odots c^{n+1})=(c^0\ps{1}\ot c^{n+1}\ps{2})\ot(c^0\ps{2}\ot c^1\odots c^n\ot c^{n+1}\ps{1})$, and on $\Cc^e\ot\Cc^{\ot\,n}$ by $\nb((c\ot c')\ot(c^1\odots c^n))=(c\ps{1}\ot c'\ps{2})\ot(c\ps{2}\ot c'\ps{1})\ot(c^1\odots c^n)$, yields
\begin{equation*}
\left(\CH^\ast(\Cc,V),\,b\right) \cong \left(V\Box_{\Cc^e}\CB^\ast(\Cc),\,d\right).
\end{equation*}
Hence, in case $\Cc$ is counital one can interpret the Hochschild cohomology of $\Cc$, with coefficients in $V$, in terms of $\cotor$-groups as
\begin{equation*}
HH^\ast(\Cc,V) = H_\ast(\CH^\ast(\Cc,V),\,b) = \cotor^\ast_{\Cc^e}(V,\Cc),
\end{equation*}
or more generally,
\begin{equation*}
HH^\ast(\Cc,V) = H_\ast(V\Box_{\Cc^e}\wbar{Y}(\Cc),\,d)
\end{equation*}
for any injective resolution $\wbar{Y}(\Cc)$ of $\Cc$ via left $\Cc^e$-comodules.

\subsection{Hopf-cyclic cohomology of Hopf algebras}

In this subsection we recall the basics of the Hopf-cyclic cohomology theory for Hopf algebras from \cite{ConnMosc98,ConnMosc}. To this end we start with the coefficient spaces for this homology theory, the stable anti-Yetter-Drinfeld (SAYD) modules.

\medskip

Let $\Hc$ be a Hopf algebra. A right $\Hc$-module, left $\Hc$-comodule $V$ is called an anti-Yetter-Drinfeld (AYD) module over $\Hc$ if
\begin{equation*}
\Db(v\cdot h)= S(h\ps{3})v\ns{-1}h\ps{1}\ot v\ns{0}\cdot h\ps{2},
\end{equation*}
for any $v\in V$, and any $h\in \Hc$, and $V$ is called stable if
\begin{equation*}
v\ns{0}\cdot v\ns{-1} = v
\end{equation*}
for any $v\in V$. In particular, the field $k$, regarded as an $\Hc$-module by a character $\d:\Hc\lra k$, and a $\Hc$-comodule via a group-like $\s\in\Hc$, is an AYD module over $\Hc$ if
\begin{equation*}
S_\d^2 = \Ad_\s, \qquad S_\d(h)=\d(h\ps{1})S(h\ps{2}),
\end{equation*}
and it is stable if
\begin{equation*}
\d(\s)=1.
\end{equation*}
Such a pair $(\d,\s)$ is called a modular pair in involution (MPI), \cite{ConnMosc00,HajaKhalRangSomm04-I}.

\medskip

Let $V$ be a right-left SAYD module over a Hopf algebra $\Hc$. Then
\begin{equation*}
C^\ast(\Hc,V)=\bigoplus_{n\geq 0}C^n(\Hc,V),\qquad C^n(\Hc,V):= V\ot \Hc^{\ot n}
\end{equation*}
is a cocyclic module \cite{HajaKhalRangSomm04-II}, via the face operators
\begin{align*}
&d_i: C^n(\Hc,V)\ra C^{n+1}(\Hc,V), \quad 0\le i\le n+1\\
&d_0(v\ot h^1\odots h^n)=v\ot 1\ot h^1\odots h^n,\\
&d_i(v\ot h^1\odots h^n)=v\ot h^1\odots h^i\ps{1}\ot h^i\ps{2}\odots h^n, \\
&d_{n+1}(v\ot h^1\odots h^n)=v\ns{0}\ot h^1\odots h^n\ot v\ns{-1},
\end{align*}
the degeneracy operators
\begin{align*}
&s_j: C^n(H,V)\ra C^{n-1}(H,V), \quad 0\le j\le n-1 \\
&s_j (v\ot h^1\odots h^n)= v\ot h^1\odots \ve(h^{j+1})\odots h^n,
\end{align*}
and the cyclic operator
\begin{align*}
&t: C^n(H,V)\ra C^n(H,V),\\
&t(v\ot h^1\odots h^n)=v\ns{0} \cdot h^1\ps{1}\ot S(h^1\ps{2})\cdot(h^2\odots h^n\ot v\ns{-1}).
\end{align*}
The total cohomology of the associated first quadrant bicomplex $\left(CC^{\ast,\ast}(\Hc,V),b,B\right)$, \cite{MoscRang09}, where
\begin{equation*}
CC^{p,q}(\Hc,V):= \left\{\begin{array}{cc}
                           C^{q-p}(\Hc,V) & if\,\,q\geq p \geq 0, \\
                           0 & if\,\, p>q,
                         \end{array}
\right.
\end{equation*}
with the coalgebra Hochschild coboundary
\begin{equation*}
b: CC^{p,q}(\Hc,V)\lra CC^{p,q+1}(\Hc,V), \qquad b:=\sum_{i=0}^{q}(-1)^i d_i,
\end{equation*}
and the Connes boundary operator
\begin{equation*}
B: CC^{p,q}(\Hc,V)\lra CC^{p-1,q}(\Hc,V), \qquad B:=\left(\sum_{i=0}^{p}(-1)^{ni}t^{i}\right) s_{p-1}t,
\end{equation*}
is called the Hopf-cyclic cohomology of the Hopf algebra $\Hc$ with coefficients in the SAYD module $V$, and is denoted by $HC(\Hc,V)$.

\medskip

Finally, the periodic Hopf-cyclic cohomology is defined similarly as the total complex of the bicomplex
\begin{equation*}
CC^{p,q}(\Hc,V):= \left\{\begin{array}{cc}
                           C^{q-p}(\Hc,V) & if\,\,q\geq p \\
                           0 & if\,\, p>q,
                         \end{array}
\right.
\end{equation*}
and is denoted by $HP(\Hc,V)$.

\section{The cohomological machinery}\label{section-cohom-mach}

This section contains the main computational tool of the present paper, namely, given a coalgebra coextension $C\lra D$ with a coflatness condition, we compute the coalgebra Hochschild cohomology of $C$ by means of the Hochschild cohomology of $D$ - on the $E_1$-term of a spectral sequence.

\medskip

Let a coextension $\pi:C\lra D$ be given. We first introduce the auxiliary coalgebra $Z:= C\oplus D$ with the comultiplication
\begin{equation*}
\Delta(y) = y_{(1)}\otimes y_{(2)} \quad\text{ and }\quad
   \Delta(x) = x_{(1)}\otimes x_{(2)} + \pi(x_{(1)})\otimes x_{(2)} + x_{(1)}\otimes \pi(x_{(2)})
\end{equation*}
and the counit
\begin{equation*}
\ve(x+y) = \ve(y),
\end{equation*}
for any $x\in C$ and $y\in D$.

\medskip

Next, let $V$ be a $C$-bicomodule, and let $C$ be coflat both as a left and a right $D$-comodule. Then consider the decreasing filtration
\begin{equation*}
G^{p+q}_p =\left\{ \begin{array}{ll}
                     \bigoplus_{n_0+\cdots+n_p=q} V\otimes Z^{\otimes n_0}\otimes D\otimes\cdots\otimes Z^{\otimes n_{p-1}}\otimes D\otimes Z^{\otimes n_p}, & p\geq 0 \\
                     0, & p<0.
                   \end{array}
\right.
\end{equation*}
In the associated spectral sequence we get
\begin{equation*}
E^{i,j}_0 = G^{i+j}_i/G^{i+j}_{i+1} =
   \bigoplus_{n_0+\cdots+n_i=j}
   V\otimes C^{\otimes n_0}\otimes D\otimes\cdots\otimes C^{\otimes n_{i-1}}\otimes D\otimes C^{\otimes n_i},
\end{equation*}
which gives us
\begin{equation*}
 E^{0,j}_1 =   HH^j(C,V).
\end{equation*}
On the horizontal differential however, by the definition of the filtration we use only the $D$-bicomodule structure on $C$. Hence, by the coflatness assumption
\begin{equation*}
E^{i,j}_2=0, \qquad i>0.
\end{equation*}
As a result, the spectral sequence collapses and we get
\begin{equation}\label{aux-Hoch-equal}
HH^n(Z,V)\cong HH^n(C,V), \qquad n\geq 0.
\end{equation}
Alternatively, one can use the short exact sequence
\begin{equation*}
  0 \to D \xrightarrow{\ i\ } Z \xrightarrow{\ p\ } C \to 0
  \qquad \text{ where }
  \qquad i: y\mapsto (0,y)
  \qquad p: (x,y)\mapsto x
\end{equation*}
of coalgebras, and \cite[Lemma 4.10]{FariSolo00}, to conclude \eqref{aux-Hoch-equal}.

\medskip

Now consider $\CH^\ast(Z,V)$, this time with the decreasing filtration
\begin{equation*}
F^{n+p}_p =  \left\{ \begin{array}{ll}
          \bigoplus_{n_0+\cdots+n_p=n} V\otimes Z^{\otimes n_0}\otimes C\otimes\cdots\otimes Z^{n_{p-1}}\otimes C\otimes Z^{\otimes n_p}, & p\geq 0 \\
          0, & p<0.
        \end{array}
\right.
\end{equation*}
The associated spectral sequence is
\begin{equation*}
E^{i,j}_0 = F^{i+j}_i / F^{i+j}_{i+1}
   = \bigoplus_{n_0+\cdots+n_i=j} V\otimes D^{\otimes n_0}\otimes C\otimes\cdots\otimes D^{n_{i-1}}\otimes C\otimes D^{\otimes n_i},
\end{equation*}
and by the coflatness assumption, on the vertical direction it computes
\begin{equation*}
HH^j(D,C^{\Box_D\,i}\,\Box_D\,V).
\end{equation*}

We can summarize our discussion in the following theorem.

\begin{theorem}\label{thm-spec-seq-main}
Let $\pi:C\lra D$ be a coalgebra projection, $V$ a $C$-bicomodule, and $C$ be coflat both as a left and a right $D$-comodule. Then there is a spectral sequence, whose $E_1$-term is
\begin{equation*}
E_1^{i,j}= HH^j(D, C^{\Box_D\,i}\,\Box_D\,V),
\end{equation*}
converging to $HH^{i+j}(C,V)$.
\end{theorem}

\section{Computations}

In this section we will apply the cohomological machinery developed in Section \ref{section-cohom-mach} to compute the Hopf-cyclic cohomology groups of quantized enveloping algebras, Connes-Moscovici Hopf algebra $\Hc_1$ and its Schwarzian quotient $\Hc_{\rm 1S}$.

\medskip

To this end, we will compute the $\cotor$-groups with MPI coefficients, from which we will obtain coalgebra Hochschild cohomology groups in view of \cite[Lemma 5.1]{Crai02}. Therefore we note the following analogue of Theorem \ref{thm-spec-seq-main}.

\begin{theorem}\label{thm-spec-seq}
Let $\pi:C\lra D$ be a coalgebra projection, and $V=V'\ot V''$ a $C$-bicomodule such that the left $C$-comodule structure is given by $V'$ and the right $C$-comodule structure is given by $V''$. Let also $C$ be coflat both as a left and a right $D$-comodule. Then there is a spectral sequence, whose $E_1$-term is of the form
\begin{equation*}
E_1^{i,j}= \cotor^j_D(V'', C^{\Box_D\,i}\,\Box_D\,V'),
\end{equation*}
converging to $HH^{i+j}(C,V)$.
\end{theorem}

\begin{proof}
The proof follows from Theorem \ref{thm-spec-seq-main} in view of the definition \eqref{aux-cobar} of the coboundary map of a cobar complex, and \eqref{aux-cotor-by-cobar}.
\end{proof}

Let us also recall the principal coextensions from \cite{Schn90}, see also \cite{BrzeHaja09}. Let $H$ be a Hopf algebra with a bijective antipode, and $C$ a left $H$-module coalgebra. Moreover, $H^+$ being the augmentation ideal of $H$ (that is, $H^+:=\ker\ve$), define the (quotient) coalgebra $D:=C/H^{+}C$.

\medskip

Then, by \cite[Theorem II]{Schn90},
\begin{itemize}
\item [(a)] $C$ is a projective left $H$-module,
\item [(b)] ${\rm can}:H\ot C\lra C\Box_D C$, $h\ot c\mapsto h\cdot c\ps{1}\ot c\ps{2}$ is injective,
\end{itemize}
if and only if
\begin{itemize}
\item [(a)] $C$ is faithfully flat left (and right) $D$-comodule,
\item [(b)] ${\rm can}:H\ot C\lra C\Box_D C$ is an isomorphism.
\end{itemize}

We will use this set-up to meet the hypothesis of Theorem \ref{thm-spec-seq-main} (and Theorem \ref{thm-spec-seq}).

\subsection{Connes-Moscovici Hopf algebras}

In this subsection we will compute the periodic Hopf-cyclic cohomology groups of the Connes-Moscovici Hopf algebra $\Hc_1$ and its Schwarzian quotient $\Hc_{\rm 1S}$ using the spectral sequence introduced in Theorem \ref{thm-spec-seq}, and the Cartan homotopy formula developed in \cite{MoscRang07}.

\medskip

We will use our main machinery to recover the results of \cite{MoscRang07} on the periodic Hopf-cyclic cohomology of the Connes-Moscovici Hopf algebra $\Hc_1$. Therefore, in this subsection we will take a quick detour to the Hopf algebra $\Hc_1$ of codimension 1, and its Schwarzian quotient $\Hc_{\rm 1S}$ from \cite{ConnMosc98,ConnMosc,MoscRang07}.

\medskip

Let $F\Rb\lra \Rb$ be the frame bundle over $\Rb$, equipped with the flat connection whose fundamental vertical vector field is
\begin{equation*}
Y = y\frac{\part}{\part y},
\end{equation*}
and the basic horizontal vector field is
\begin{equation*}
X = y\frac{\part}{\part x},
\end{equation*}
in local coordinates of $F\Rb$. They act on the crossed product algebra $\Ac=C^\infty_c(F\Rb)\rtimes \Diff(\Rb)$, a typical element of which is written by $fU^\ast_\vp := f\rtimes \vp^{-1}$, via
\begin{equation*}
Y(fU^\ast_\vp) = Y(f)U^\ast_\vp, \qquad X(fU^\ast_\vp) = X(f)U^\ast_\vp.
\end{equation*}
Then $\Hc_1$ is the unique Hopf algebra that makes $\Ac$ to be a (left) $\Hc_1$-module algebra. To this end one has to introduce the further differential operators
\begin{equation*}
\d_n(fU^\ast_\vp):=y^n\frac{d}{d x^n}(\log \, \vp'(x))fU^\ast_\vp,\qquad n\geq 1,
\end{equation*}
and the Hopf algebra structure of $\Hc_1$ is given by
\begin{align*}
\begin{split}
& [Y,X]=X,\quad [Y,\d_n]=n\d_n,\quad [X,\d_n]=\d_{n+1},\qquad [\d_n,\d_m]=0, \\
& \D(Y) = Y\ot 1+1\ot Y, \\
& \D(\d_1) = \d_1\ot 1+1\ot \d_1, \\
& \D(X) = X\ot 1+1\ot X + \d_1\ot Y, \\
& \ve(X)=\ve(Y)=\ve(\d_n)=0,\\
& S(X) = -X +\d_1Y,\quad S(Y)=-Y,\quad S(\d_1)=-\d_1.
\end{split}
\end{align*}
The ideal generated by the Schwarzian derivative
\begin{equation*}
\d_2':=\d_2-\frac{1}{2}\d_1^2,
\end{equation*}
is a Hopf ideal (an ideal, a coideal and is stable under the antipode), therefore the quotient $\Hc_{\rm 1S}$ becomes a Hopf algebra, called the Schwarzian Hopf algebra. As an algebra $\Hc_{\rm 1S}$ is generated by $X,Y,Z$, and the Hopf algebra structure is given by
\begin{align*}
\begin{split}
& [Y,X]=X,\quad [Y,\d_n]=n\d_n,\quad [X,Z]=\frac{1}{2}Z^2, \\
& \D(Y) = Y\ot 1+1\ot Y, \\
& \D(Z) = Z\ot 1+1\ot Z, \\
& \D(X) = X\ot 1+1\ot X + Z\ot Y, \\
& \ve(X)=\ve(Y)=\ve(Z)=0,\\
& S(X) = -X +ZY,\quad S(Y)=-Y,\quad S(Z)=-Z.
\end{split}
\end{align*}
Hence
\begin{equation*}
\Fc := {\rm Span}\left\{\d_{\a_1}^{n_1}\ldots\d_{\a_p}^{n_p}\,|\,p,\a_1,\ldots\a_p\geq 1,\,n_1,\ldots n_p\geq 0\right\} \subseteq \Hc_1
\end{equation*}
is a Hopf subalgebra of $\Hc_1$. Finally let us note, by \cite{RangSutl-III}, that the only modular pair in involution (MPI) on $\Hc_1$ is $(\d,1)$ of \cite{ConnMosc98}, where $\d:g\ell_1^{\rm aff}\lra k$ is the trace of the adjoint representation of $g\ell_1^{\rm aff}$ on itself.

\medskip

Let $C:=\Hc_1$, and let us consider the Hopf subalgebra $\Fc\subseteq C$. Then
\begin{equation*}
D=C/\Fc^+C=\Uc:=U(g\ell_1^{\rm aff}),
\end{equation*}
see \cite[Lemma 3.19]{MoscRang09}, or \cite[Section 5]{ConnMosc}. Hence, by \cite[Thm. II]{Schn90} we conclude that $C$ is (faithfully) coflat as left and right $D$-comodule. Therefore, the hypothesis of Theorem \ref{thm-spec-seq} is satisfied.

\medskip

Since $\,^\s k = k$, we have
\begin{equation*}
C^{\Box_D\,i}\,\Box_D\,k = \Fc^{\ot\,i},
\end{equation*}
and hence by Theorem \ref{thm-spec-seq-main},
\begin{equation*}
E_1^{i,j} = HH^j(\Uc,\Fc^{\ot\,i}) \Rightarrow HH^{i+j}(\Hc_1,k).
\end{equation*}
Moreover, since the $\Uc$-coaction on $\Fc$ is trivial, we have
\begin{equation*}
E_1^{i,j} = HH^j(\Uc,k)\ot \Fc^{\ot\,i} \Rightarrow HH^{i+j}(\Hc_1,k).
\end{equation*}

As a result of the Cartan homotopy formula \cite[Coroll. 3.9]{MoscRang07} for $\Hc_1$, one has \cite[Coroll. 3.10]{MoscRang07} as recalled below. We adopt the same notation from \cite{MoscRang07}, and we denote by
$HP(\Hc_{1\natural}[p],k)$ the periodic Hopf-cyclic cohomology, with trivial coefficients, of the weight $p$ subcomplex of $\Hc_1$, with respect to the grading given by the adjoint action of $Y\in \Hc_1$.

\begin{corollary}
The periodic Hopf-cyclic cohomology groups of $\Hc_1$ are computed by the weight 1 subcomplex, \ie
\begin{equation*}
HP(\Hc_{1\natural}[1],k) = HP(\Hc_1,k),\qquad HP(\Hc_{1\natural}[p],k) = 0,\quad p\neq 1.
\end{equation*}
\end{corollary}

Since our spectral sequence respects the weight, in view of \cite[Thm. 18.7.1]{Kass-book} we check only
\begin{align*}
\begin{split}
& E_1^{1,0} = \big\langle\one\ot\d_1\ot\one\big\rangle \in k\ot C\,\Box_D\,k,\\
& E_1^{1,1} = \big\langle\one\ot \wbar{Y}\ot\d_1\ot\one\big\rangle \in k\ot D\ot C\,\Box_D\,k,\\
& E_1^{0,1} = \big\langle\one\ot\wbar{X}\ot\one\big\rangle \in k\ot D\ot\,k,\\
& E_1^{0,2} = \big\langle\one\ot\wbar{X\wg Y}\ot\one\big\rangle \in k\ot D\ot D\ot k,
\end{split}
\end{align*}
as the weight 1 subcomplex. Here by $\wbar{x}\in D$ we mean the element $x\in\Hc_1$ viewed in $D$.

\medskip

Let $d_0:E_0^{i,j}\lra E_0^{i,j+1}$ be the vertical, and $d_1:E_1^{i,j}\lra E_1^{i+1,j}$ be the horizontal coboundary. Then we first have
\begin{equation*}
d_1\left(\one\ot\d_1\ot\one\right) = \one\ot 1 \ot \d_1\ot\one- \one\ot \D(\d_1)\ot\one + \one\ot\d_1\ot 1 \ot\one=0.
\end{equation*}
Next, we similarly observe
\begin{equation*}
d_1\left(\one\ot\wbar{X}\ot\wbar{Y}\ot\one\right) = \one\ot 1\ot \wbar{X}\ot\wbar{Y}\ot\one- \one\ot \wbar{X}\ot\wbar{Y}\ot 1\ot \one,
\end{equation*}
and
\begin{align*}
\begin{split}
& d_0\Big(\one\ot X\ot\wbar{Y}\ot\one+ \one\ot \wbar{X}\ot Y\ot 1\ot \one+\frac{1}{2}\one\ot \d_1\ot \wbar{Y}^2\ot \one\Big) = \\
& \one\ot \wbar{X}\ot\wbar{Y}\ot 1\ot \one - \one\ot 1\ot \wbar{X}\ot\wbar{Y}\ot\one.
\end{split}
\end{align*}
On the other hand, we have
\begin{equation*}
d_1\left(\one\ot\wbar{Y}\ot\wbar{X}\ot\one\right) = \one\ot 1\ot \wbar{Y}\ot\wbar{X}\ot\one- \one\ot \wbar{Y}\ot\wbar{X}\ot 1\ot \one,
\end{equation*}
and
\begin{align*}
\begin{split}
& d_0\Big(\one\ot Y\ot\wbar{X}\ot\one+ \one\ot \wbar{Y}\ot X\ot \one \\
& \hspace{3cm} -\frac{1}{2}\one\ot  \wbar{Y}^2\ot\d_1\ot \one - \one\ot  \wbar{Y}\ot\d_1Y\ot \one \Big) = \\
& \one\ot \wbar{Y}\ot\wbar{X}\ot 1\ot \one - \one\ot 1\ot \wbar{Y}\ot\wbar{X}\ot\one.
\end{split}
\end{align*}
Therefore
\begin{equation*}
d_1\left(\one\ot\wbar{X\wg Y}\ot\one\right) = 0.
\end{equation*}
Finally we calculate
\begin{equation*}
d_1\left(\one\ot\wbar{X}\ot\one\right) = \one\ot 1\ot \wbar{X}\ot\one + \one\ot\wbar{X}\ot 1 \ot \one.
\end{equation*}
We also note that
\begin{equation*}
d_0\left(\one\ot X\ot\one\right) = - \one\ot \wbar{X}\ot 1 \ot \one - \one\ot 1 \ot \wbar{X} \ot \one - \one\ot \d_1 \ot \wbar{Y} \ot \one,
\end{equation*}
and
\begin{equation*}
d_0\left(\one\ot \d_1Y\ot\one\right) = - \one\ot \wbar{Y}\ot \d_1 \ot \one -  \one\ot \d_1 \ot \wbar{Y} \ot \one.
\end{equation*}
Hence,
\begin{equation*}
\one\ot \wbar{Y}\ot \d_1 \ot \one = d_1\left(\one\ot\wbar{X}\ot\one\right)+ d_0\left(\one\ot X\ot\one -\one\ot \d_1Y\ot\one\right).
\end{equation*}
As a result, on the $E_2$-term we will see
\begin{equation*}
E_2^{1,0} = \big\langle\one\ot\d_1\ot\one\big\rangle,\qquad E_2^{0,2} = \big\langle\one\ot\wbar{X\wg Y}\ot\one\big\rangle.
\end{equation*}

Transgression of these cocycles yields \cite[Prop. 4.3]{MoscRang07} as follows.
\begin{proposition}\label{prop-weight-1-Hochschild}
The Hochschild cohomology of the weight 1 subcomplex of $\Hc_1$ is generated by
\begin{equation*}
[\d_1]\in HH^1(\Hc_1,k),\qquad [X\ot Y - Y\ot X - \d_1Y\ot Y]\in HH^2(\Hc_1,k).
\end{equation*}
\end{proposition}

Consequently, we recover \cite[Thm. 4.4]{MoscRang07}.
\begin{theorem}
The periodic Hopf-cyclic cohomology of $\Hc_1$ with coefficients in the SAYD module ${}^1k_{\ve}$ is given by
\begin{equation*}
HP^{\rm odd}(\Hc_1,k) = \big\langle\d_1\big\rangle,\qquad HP^{\rm even}(\Hc_1,k) = \big\langle X\ot Y - Y \ot X -\d_1Y\ot Y\big\rangle.
\end{equation*}
\end{theorem}

On the Schwarzian quotient we similarly recover \cite[Thm. 4.5]{MoscRang07} as follows.
\begin{theorem}
The periodic Hopf-cyclic cohomology of $\Hc_{\rm 1S}$ with coefficients in the SAYD module ${}^1k_{\ve}$ is given by
\begin{equation*}
HP^{\rm odd}(\Hc_{\rm 1S},k) = \big\langle Z\big\rangle,\qquad HP^{\rm even}(\Hc_{\rm 1S},k) = \big\langle X\ot Y - Y \ot X -ZY\ot Y\big\rangle.
\end{equation*}
\end{theorem}

\subsection{Quantum enveloping algebras}

In this subsection we will compute the (periodic and non-periodic)
Hopf-cyclic cohomology groups of the quantized enveloping algebras
$U_q(\Fg)$. Our strategy will be to realize it as a principal
coextension.

\medskip

Let us first recall Drinfeld-Jimbo quantized
enveloping algebras of Lie algebras from \cite[Subsect. 6.1.2]{KlimSchm-book}.

\medskip

Let $\Fg$ be a finite dimensional complex semi-simple Lie algebra, $\a_1,\ldots,\a_\ell$ a fixed ordered sequence of simple roots, and $A=[a_{ij}]$ the Cartan matrix. Let also $q$ be a fixed nonzero complex number such that $q_i^2\neq 1$, where $q_i:= q^{d_i}$, $1\leq i\leq \ell$, and $d_i=(\a_i,\a_i)/2$.

\medskip

Then the Drinfeld-Jimbo quantized enveloping algebra $U_q(\Fg)$ is the Hopf algebra with $4\ell$ generators $E_i,F_i,K_i,K_i^{-1}$, $1\leq i\leq \ell$, and the relations
\begin{align*}
& K_iK_j = K_jK_i,\qquad K_iK_i^{-1}=K_i^{-1}K_i=1,\\
& K_iE_jK_i^{-1}=q_i^{a_{ij}}E_j,\qquad K_iF_jK_i^{-1}=q_i^{-a_{ij}}F_j,\\
& E_iF_j-F_jE_i=\d_{ij}\frac{K_i-K_i^{-1}}{q_i-q_i^{-1}},\\
& \sum_{r=0}^{1-a_{ij}}(-1)^r\left[\begin{array}{c}
                                              1-a_{ij} \\
                                              r
                                            \end{array}
\right]_{q_i}E_i^{1-a_{ij}-r}E_jE_i^r=0, \quad i\neq j,\\
& \sum_{r=0}^{1-a_{ij}}(-1)^r\left[\begin{array}{c}
                                              1-a_{ij} \\
                                              r
                                            \end{array}
\right]_{q_i}F_i^{1-a_{ij}-r}F_jF_i^r=0, \quad i\neq j,
\end{align*}
where
\begin{equation*}
\left[\begin{array}{c}
        n \\
        r
      \end{array}
\right]_q = \frac{(n)_q\,!}{(r)_q\,!\,\,(n-r)_q\,!},\qquad (n)_q:=\frac{q^n-q^{-n}}{q-q^{-1}}.
\end{equation*}
The rest of the Hopf algebra structure of $U_q(\Fg)$ is given by
\begin{align*}
& \D(K_i)=K_i\ot K_i,\quad \D(K_i^{-1})=K_i^{-1}\ot K_i^{-1} \\
& \D(E_i)=E_i\ot K_i + 1\ot E_i,\quad \D(F_j)=F_j\ot 1 + K_j^{-1}\ot F_j \\
& \ve(K_i)=1,\quad \ve(E_i)=\ve(F_i)=0\\
& S(K_i)=K_i^{-1},\quad S(E_i)=-E_iK_i^{-1},\quad S(F_i)=-K_iF_i.
\end{align*}
Let us also recall, from \cite{KlimSchm-book}, the Hopf-subalgebras
\begin{align*}
\begin{split}
& U_q(\Fb_+)={\rm Span}\left\{E^{p_1}_1\ldots E_\ell^{p_\ell}K^{q_1}_1\ldots K_\ell^{q_\ell}\,|\,r_1,\ldots r_\ell\geq 0,\,q_1,\ldots, q_\ell \in \Zb\right\},\\
& U_q(\Fb_-)={\rm Span}\left\{K^{q_1}_1\ldots K_\ell^{q_\ell}F^{r_1}_1\ldots F_\ell^{r_\ell}\,|\,p_1,\ldots p_\ell\geq 0,\,q_1,\ldots, q_\ell \in \Zb\right\},
\end{split}
\end{align*}
of $U_q(\Fg)$.

\medskip

A modular pair in involution for the Hopf algebra $U_q(\Fg)$ is given by \cite[Prop. 6.6]{KlimSchm-book}. Let $K_\lambda:=K_1^{n_1}\ldots K_\ell^{n_\ell}$ for any $\lambda=\sum_in_i\a_i$, where $n_i\in \Zb$. Then, $\rho \in \Fh^\ast$ being the half-sum of the positive roots of $\Fg$, by \cite[Prop. 6.6]{KlimSchm-book} we have
\begin{equation*}
S^2(a) = K_{2\rho}aK^{-1}_{2\rho},\qquad \forall a\in U_q(\Fg).
\end{equation*}
Thus, $(\ve,K_{2\rho})$ is a MPI for the Hopf algebra $U_q(\Fg)$.

\medskip

For the Hopf subalgebra $H:=U_q(\Fb_+) \subseteq U_q(\Fg)=:C$, we obtain $D=C/CH^+=U_q(\Fb_-)$. Then by \cite[Thm. II]{Schn90} we conclude that $C$ is (faithfully) coflat as left and right $D$-comodule, and hence the hypothesis of Theorem \ref{thm-spec-seq} is satisfied.

\begin{lemma}\label{lemma-cotor-D-K-II}
Let $C$ and $D$ be as above, and $\mu=K_1^{p_1}\ldots K_\ell^{p_\ell}$, $p_1,\ldots,p_\ell \geq 0$. Then we have
\begin{equation*}
\cotor^n_D(k,\,^\mu k) = \left\{\begin{array}{cl}
                                                                        k^{\oplus\,\frac{(p_1+\ldots+p_\ell)\,!}{p_1\,!\ldots \,p_\ell\,!}} & if\,\,n=p_1+p_2+\ldots+p_\ell, \\
                                                                        0 & if\,\,n\neq p_1+p_2+\ldots+p_\ell
                                                                      \end{array}
\right.
\end{equation*}
\end{lemma}

\begin{proof}
We apply Theorem \ref{thm-spec-seq} to the coextension
\begin{align*}
\begin{split}
& \pi:D\lra W:={\rm Span}\{K_1^{m_1}\ldots K_\ell^{m_\ell}\,|\,m_1,\ldots,m_\ell\in \Zb\}\\
& E_1^{r_1}\ldots E_\ell^{r_\ell} K_1^{m_1}\ldots K_\ell^{m_\ell}\mapsto \left\{\begin{array}{cl}
                         K_1^{m_1}\ldots K_\ell^{m_\ell} & if\,\,r_1=r_2=\ldots = r_\ell=0, \\
                         0 & otherwise,
                       \end{array}
\right.
\end{split}
\end{align*}
to have a spectral sequence, converging to $\cotor_D(k,\,^\mu k)$, whose $E_1$-term is
\begin{equation*}
E_1^{i,j} = \cotor^j_W(k,\underbrace{D\,\Box_W\,\ldots \,\Box_W\,D}_{i\,\,many}\Box_W\,^\mu k).
\end{equation*}
Since
\begin{align*}
\begin{split}
& \underbrace{D\,\Box_W\,\ldots \,\Box_W\,D}_{i\,many}\Box_W\,^\mu k =\\
& {\rm Span}\Big\{E_{a_s}^{b_s}\ldots E_{\a_s}^{\b_s}\mu K_{a_1}^{-b_1}\ldots K_{\a_1}^{-\b_1}\ldots  K_{a_s}^{-b_s}\ldots K_{\a_s}^{-\b_s}\odots\\
& \underbrace{\mu K_{a_1}^{-b_1}\ldots K_{\a_1}^{-\b_1}\odots \mu K_{a_1}^{-b_1}\ldots K_{\a_1}^{-\b_1}}_{i_2\,\,many}\ot\\
&\hspace{4cm} E_{a_1}^{b_1}\ldots E_{\a_1}^{\b_1} \mu K_{a_1}^{-b_1}\ldots K_{\a_1}^{-\b_1}\ot\underbrace{\mu\odots \mu}_{i_1\,\,many}\ot\one\Big\},
\end{split}
\end{align*}
$i\geq s\geq 0$, $\ell\geq a_1,\ldots,a_s,\ldots, \a_1,\ldots , \a_s \geq 1$, $i_1,i_2,\ldots, b_1,\ldots b_s,\b_1,\ldots,\b_s \geq 0$, the left $W$-coaction on a typical element is given by
\begin{align}\label{aux-K-coaction}
\begin{split}
& \nb_W^L(E_{a_i}K_{a_i}^{-1}\ldots K_{a_1}^{-1}\mu\odots E_{a_2}K_{a_2}^{-1}K_{a_1}^{-1}\mu\ot E_{a_1}K_{a_1}^{-1}\mu\ot\one) = \\
&K_{a_i}^{-1}\ldots K_{a_1}^{-1}\mu \ot E_{a_i}K_{a_i}^{-1}\ldots K_{a_1}^{-1}\mu\odots E_{a_2}K_{a_2}^{-1}K_{a_1}^{-1}\mu\ot E_{a_1}K_{a_1}^{-1}\mu\ot\one.
\end{split}
\end{align}
Since $W$ consists only of the group-like elements, the result follows from the $W$-coaction \eqref{aux-K-coaction} to be trivial.
\end{proof}

\begin{lemma}\label{lemma-cotor-D-C-sigma-II}
Let $C$ and $D$ be as above, and $\s=K_{2\rho}$. Then we have
\begin{equation*}
\cotor^n_D(k,\underbrace{C\,\Box_D\,\ldots \,\Box_D\,C}_{i\,\,many}\,\Box_D\,^\s k) = \left\{\begin{array}{cc}
                                                                        k^{\oplus\,\left(\begin{array}{c}
                                                                                           \ell \\
                                                                                           i
                                                                                         \end{array}
                                                                        \right)} & if\,\, n=\ell-i, \\
                                                                        0 & if\,\, n\neq \ell-i.
                                                                      \end{array}
\right.
\end{equation*}
\end{lemma}

\begin{proof}
Let us first note that
\begin{align*}
\begin{split}
& \underbrace{C\,\Box_D\,\ldots \,\Box_D\,C}_{i\,\,many}\,\Box_D\,{}^\s k = \\
& {\rm Span}\Big\{\s K_{a_1}^{-b_1}\ldots K_{\a_1}^{-\b_1}\ldots K_{a_{s-1}}^{-b_{s-1}}\ldots K_{\a_{s-1}}^{-\b_{s-1}}F_{a_s}^{b_s}\ldots F_{\a_s}^{\b_s}\odots  \\
&  \underbrace{\s K_{a_1}^{-b_1}\ldots K_{\a_1}^{-\b_1}\odots \s K_{a_1}^{-b_1}\ldots K_{\a_1}^{-\b_1}}_{i_2\,\,many}\ot \s F_{a_1}^{b_1}\ldots F_{\a_1}^{\b_1}\ot\underbrace{\s\odots \s}_{i_1\,\,many}\ot\one\Big\},
\end{split}
\end{align*}
$i\geq s\geq 0$, $\ell\geq a_1,\ldots,a_s,\ldots, \a_1,\ldots , \a_s \geq 1$, $i_1,i_2,\ldots, b_1,\ldots b_s,\b_1,\ldots,\b_s \geq 0$. Then, the left $D$-coaction on a typical element is given by
\begin{align}\label{aux-D-coact-C-D}
\begin{split}
& \nb_D^L(\s K_{a_1}^{-1}\ldots K_{a_{i-1}}^{-1} F_{a_i}\odots \s K_{a_1}^{-1} F_{a_2}\ot \s F_{a_1}\ot\one) =\\
& \s K_{a_1}^{-1}\ldots K_{a_{i-1}}^{-1}K_{a_i}^{-1} \ot  \s K_{a_1}^{-1}\ldots K_{a_{i-1}}^{-1} F_{a_i}\odots \s K_{a_1}^{-1} F_{a_2}\ot \s F_{a_1}\ot\one.
\end{split}
\end{align}
By the proof of Lemma \ref{lemma-cotor-D-K-II}, there is no repetition on the indexes appearing in \eqref{aux-D-coact-C-D}. Accordingly, the result follows from Lemma \ref{lemma-cotor-D-K-II}.
\end{proof}

\begin{proposition}\label{prop-C-Hochschild}
For $\s:=K_{2\rho}$ we have
\begin{equation*}
\cotor_{U_q(\Fg)}^n(k,\,^\s k)=\left\{\begin{array}{cc}
                            k^{\oplus \,2^\ell} & n=\ell \\
                            0 & n\neq \ell.
                          \end{array}
\right.
\end{equation*}
\end{proposition}

\begin{proof}
Letting $C$ and $D$ be as above, it follows from Theorem \ref{thm-spec-seq} that we have a spectral sequence converging to $\cotor_C(k,\,^\s k)$ whose $E_1$-term is
\begin{equation*}
E_1^{i,j}=\cotor^j_D(k,C^{\Box_D\,i}\,\Box_D\,^\s k).
\end{equation*}
By Lemma \ref{lemma-cotor-D-C-sigma-II} the cocycles are computed at $j=\ell-i$, \ie $i+j=\ell$. The number of cocycles, on the other hand, is
\begin{equation*}
\sum_{i=0}^\ell\left(\begin{array}{c}
                       \ell \\
                       i
                     \end{array}
\right) =2^\ell.
\end{equation*}
\end{proof}

\begin{remark}
{\rm
We note that since $C=U_q(\Fg)$ is a Hopf algebra, we could start with the Hopf subalgebra $H=C\subseteq C$ to get $D=H/H^+=W$, and yet to arrive at the same result.
}
\end{remark}

We are now ready to compute the (periodic) Hopf-cyclic cohomology of $U_q(\Fg)$.

\begin{theorem}\label{thm-Uq-g-hopf-cyclic}
For $\s:=K_{2\rho}$, and $\ell \equiv \epsilon \, ({\rm mod}\,2)$, we have
\begin{equation*}
HP^\epsilon(U_q(\Fg),\,^\s k) = k^{\oplus \,2^\ell},\qquad HP^{1-\epsilon}(U_q(\Fg),\,^\s k) = 0.
\end{equation*}
\end{theorem}

\begin{proof}
Let $C$ be as above. As a result of Proposition \ref{prop-C-Hochschild} we have
\begin{equation*}
HH^n(C,\,^\s k)=\left\{\begin{array}{cc}
                            k^{\oplus \,2^\ell} & n=\ell \\
                            0 & n\neq \ell.
                          \end{array}
\right.
\end{equation*}
Hence, the Connes' SBI sequence yields
\begin{align*}
\begin{split}
& HC^{n}(C,\,^\s k) = 0,\qquad n<\ell,\\
& HC^{\ell+1}(C,\,^\s k) \cong HC^{\ell+3}(C,\,^\s k) \cong \ldots \cong 0 \\
& k^{\oplus\,2^\ell}\cong HH^{\ell}(C,\,^\s k) \cong HC^{\ell}(C,\,^\s k) \cong HC^{\ell+2}(C,\,^\s k) \cong \ldots
\end{split}
\end{align*}
where the isomorphisms are given by the periodicity map
\begin{equation*}
S:HC^{p}(C,\,^\s k)\lra HC^{p+2}(C,\,^\s k).
\end{equation*}
\end{proof}

{\bf Acknowledgement:} S. S\"utl\"u would like to thank Institut des Hautes \'Etudes Scientifiques (IHES) for the hospitality, and the stimulating environment provided during the preparation of this work.

\bibliographystyle{amsplain}
\bibliography{references}{}

\end{document}